\documentclass[10pt,a4paper]{article}
\usepackage{amssymb,amsmath,amsfonts,amsthm,authblk,hyperref,here,enumerate,
xfrac,mathtools,graphicx,pstool,xcolor,pgf,tikz,array,fancyhdr,caption}

\usepackage[utf8]{inputenc}
\usepackage{mlmodern}
\usepackage{soul} %<- Pour barrer du texte

\usepackage{tocloft}
\usepackage{subfig}
\usepackage[all]{xy}
\usetikzlibrary{arrows,arrows.meta}
\usepackage[english]{babel}
\usepackage{titlesec}
\usepackage[expansion=true]{microtype}
\usepackage[utf8]{inputenc}
\usepackage[sort&compress]{natbib}
\DeclareMathOperator*{\argmax}{arg\,max}
\numberwithin{equation}{section}
\newcommand{\R}{\mathbb R}
\newcommand{\N}{\mathbb N}

\definecolor{darkgreen}{rgb}{0,.6,0}

\long\def\@makefnmark{%
        \hbox {\@textsuperscript {\@thefnmark}}
        }

\usepackage{caption}
\titleformat{\section}
  {\normalfont\Large\bfseries}
  {\thesection}{1em}{}
  \titleformat{\subsection}
  {\normalfont\bfseries\large}
  {\thesubsection}{1em}{}
    \titleformat{\subsubsection}
  {\normalfont\bfseries\normalsize}
  {\thesubsubsection}{1em}{}
\definecolor{MyDarkBlue}{rgb}{0,0.29,0.7}
\hypersetup{
    colorlinks=true,       % false: boxed links; true: colored links
    linkcolor=MyDarkBlue,          % color of internal links (change box color with linkbordercolor)
    linkbordercolor=MyDarkBlue,
    citecolor=MyDarkBlue,
     citebordercolor=MyDarkBlue,        % color of links to bibliography
    filecolor=MyDarkBlue,      % color of file links
    urlcolor=MyDarkBlue           % color of external links
}

\newtheoremstyle{plain}% name of the style to be used
  {10pt}% measure of space to leave above the theorem. E.g.: 3pt
  {10pt}% measure of space to leave below the theorem. E.g.: 3pt
  {\it }% name of font to use in the body of the theorem
  {0pt}% measure of space to indent
  {\bf }% name of head font
  {}% punctuation between head and body
  {\newline}% space after theorem head; " " = normal interword space
  {}
\newtheoremstyle{definition}% name of the style to be used
  {10pt}% measure of space to leave above the theorem. E.g.: 3pt
  {10pt}% measure of space to leave below the theorem. E.g.: 3pt
  {}% name of font to use in the body of the theorem
  {0pt}% measure of space to indent
  {\bf }% name of head font
  {}% punctuation between head and body
  {\newline}% space after theorem head; " " = normal interword space
  {}

\theoremstyle{plain}

\newtheorem{theorem}{Theorem}[section]

\newtheorem{prop}[theorem]{Proposition}
\newtheorem*{theorem*}{Theorem}
\newtheorem*{lemma*}{Lemma}
\theoremstyle{definition}

\newtheorem*{example}{Example}
\newtheorem*{remark}{Remark}

\usepackage{xpatch}
\makeatletter
\xpatchcmd{\@makefnstartbox}{%
\footnotesize}%
{\footnotesize\sffamily}{}{}
\xpatchcmd{\@makefnmark}{\normalfont}{\sffamily}{}{}
\makeatother
  
\fancypagestyle{plain}{%
\fancyhf{} % clear all header and footer fields
\fancyfoot[C]{\fontsize{10pt}{10pt}\thepage} % except the center

}
\begin{document}

\title{Counting Nodes in Smolyak Grids}
\author[1]{Jocelyn Minini}
\author[1,2]{Micha Wasem}
\affil[1]{\normalsize School of Engineering and Architecture, HES-SO University of Applied Sciences and Arts Western Switzerland, P\'erolles 80, 1700 Fribourg, Switzerland\\{\tt jocelyn.minini@hefr.ch}, {\tt micha.wasem@hefr.ch}}
\affil[2]{\normalsize Faculty of Mathematics and Computer Science, UniDistance Suisse, Schinerstrasse 18, 3900 Brig, Switzerland\\{\tt micha.wasem@fernuni.ch}}
\date{March 29, 2025}
\maketitle

\begin{abstract}
Using generating functions, we are proposing a unified approach to produce explicit formulas, which count the number of nodes in Smolyak grids based on various univariate quadrature or interpolation rules. Our approach yields, for instance, a new formula for the cardinality of a Smolyak grid, which is based on Chebyshev nodes of the first kind and it allows to recover certain counting-formulas previously found by Bungartz-Griebel, Kaarnioja, Müller-Gronbach, Novak-Ritter and Ullrich. \end{abstract}
%{\sffamily\tableofcontents}


\section{ Introduction}

The aim of this article is to fill a gap in the literature by presenting a general framework which allows to produce explicit formulas for the number of nodes in a Smolyak grid.

Let $U_i$ be a quadrature rule on $[-1,1]$, i.e.\ $U_i : C^0([-1,1])\to\mathbb R$ or an interpolation rule i.e.\ $U_i: C^0([-1,1])\to C^0([-1,1])$ which evaluates an unknown function $g:[-1,1]\to\R$ on a finite set $S_i\subset [-1,1]$. The Smolyak rule turns a sequence of univariate quadrature or interpolation rules $U_1,U_2,\ldots $ into a quadrature or interpolation rule on $[-1,1]^d$. Here, $|S_i|\leqslant |S_{i+1}|$ for all $i\in \mathbb N_{\geqslant 1}$ and the Smolyak rule in $d$ dimensions of \emph{level} $\mu\in\mathbb N$ is defined as
\begin{equation}\label{smolyak-rough}
Q^d_\mu = \sum_{d\leqslant |\mathbf i|\leqslant d+\mu}\bigotimes_{k=1}^d \Delta_{i_k},
\end{equation}
where $\Delta_{i}=U_{i}-U_{i-1}$ as long as $i>2$ and $\Delta_1 = U_1$, $\mathbf i\in \N^d_{\geqslant 1}$ is a multi-index with positive entries and $|\mathbf i|:=i_1+\ldots + i_d$ denotes its $1$-norm  (see \cite{smolyak}). \cite{kaarnioja_smolyak_2013} shows that \eqref{smolyak-rough} includes a lot of cancellation and can be expressed as (see also \citet[Lemma 1]{wasilkowski_explicit_1995})
$$
Q^d_\mu = \sum_{\max\{d,\mu+1\}\leqslant |\mathbf i|\leqslant d+\mu}(-1)^{d+\mu-|\mathbf i|}\binom{d-1}{d+\mu-|\mathbf i|}\bigotimes_{k=1}^d U_{i_k}.
$$
This last expression implies in which nodes an unknown function $u:[-1,1]^d\to\R$ needs to be evaluated, if $Q^d_{\mu}(u)$ is computed: We attach to the sequence $U_1,U_2,\ldots$ of univariate quadrature or interpolation rules the sequence $\mathcal S = S_1, S_2, \ldots $ of sets $S_i\subset[-1,1]$ of evaluation nodes such that $f(i) = |S_i|$ is a non-decreasing function $f:\N_{\geqslant 1}\to \N_{\geqslant 1}$, called \emph{growth function} of $\mathcal S$. The Smolyak grid in $d$ dimensions with \emph{level} $\mu$ is then given by
$$
\Gamma_d(\mu) = \bigcup_{\max\{d,\mu+1\}\leqslant |\mathbf{i}|\leqslant d+\mu}\bigtimes_{k=1}^d S_{i_k}.
$$
If the sequence $\mathcal S$ is \emph{nested}, i.e.\ if
$
S_k\subset S_{k+1}$ for all $k\in \mathbb N_{\geqslant 1}$, then
$$
\Gamma_d(\mu) = \bigcup_{|\mathbf{i}|= d+\mu}\bigtimes_{k=1}^d S_{i_k}.
$$
Since the Smolyak rule has been introduced to overcome the curse of dimensionality, there is great interest in determining its computational cost i.e.\ to count the number of nodes in such a grid for a given sequence $\mathcal S$, a given dimension $d$ and a level $\mu$. This number will be denoted by $\mathrm N_d(\mu)$, where the growth function $f$ of the sequence $\mathcal S$ will be clear from the context. For our purposes, we consider isotropic grids but we remark that our counting argument at no point needs spatial isotropy but only that the sets used in each dimension share the same growth function.

For practical implementation, it might be of interest to count the number of nodes needed during the generating process of such grids, before duplicates are deleted, i.e.\ there might be interest in finding a simpler formula for
$$
\mathcal N^{\Sigma}_d(\mu)=\sum_{\max\{d,\mu+1\}\leqslant |\mathbf{i}|\leqslant d+\mu}\prod_{k=1}^df(i_k),
$$
which doesn't require to compute all the multi-indices $\mathbf i\in \mathbb N_{\geqslant 1}^d$ such that $\max\{d,\mu+1\}\leqslant |\mathbf{i}|\leqslant d+\mu$. The number $\mathcal N^{\Sigma}_d(\mu)$ is of course independent of the nestedness or non-nestedness of the sequence $\mathcal S$, however, if the sequence is nested, counting points in the Smolyak grid with duplicates boils down to finding a formula for
$$
\mathcal N_d(\mu)=\sum_{|\mathbf{i}|= d+\mu}\prod_{k=1}^df(i_k).
$$
We will provide formulas for the quantities $\mathrm N_d(\mu)$ for the case of $\mathcal S$ being nested and $\mathcal N_d(\mu)$ for general $\mathcal S$ only as 
$$
\mathcal N^{\Sigma}_d(\mu) =\!\!\!\!\! \sum_{k=\max\{0,\mu+1-d\}}^\mu\!\!\!\!\! \mathcal N_d(k).
$$
%Whenever $f$ is clear from the context, we will suppress the dependence on $f$ and write for instance $\mathrm N_d(\mu):=\mathrm N_d(d,\mu,f)$.

For some specific growth functions, there are explicit formulas for $\mathrm N_d(\mu)$ available in the literature, see for instance \cite{kaarnioja_smolyak_2013,ullrich,Bungartz_Griebel_2004}. Also, a lot is known about upper bounds for $\mathrm N_d(\mu)$ or asymptotic formulas \citep{ullrich,kaarnioja_smolyak_2013,xu_weak_2015,muller-gronbach_hyperbolic_1998,wasilkowski_explicit_1995,novak_simple_1999}. The contribution of \cite{burkardt} provides tables with values for $\mathrm N_d(\mu)$ and code which allows to generate them. Also, in \cite{judd}, tables with values are provided.

Our approach allows to find closed formulas for the quantities $\mathrm N_d(\mu)$ and $\mathcal N_d(\mu)$ by using generating functions and dimension-wise induction, thus generalizing previous formulas obtained by of \cite{novak_simple_1999,muller-gronbach_hyperbolic_1998,Bungartz_Griebel_2004} and \cite{ullrich}. To this end, we define
$$
G_d(x) = \sum_{\mu=0}^\infty\mathrm N_d(\mu)x^\mu \text{ and }\mathcal G_d(x) = \sum_{\mu=0}^\infty\mathcal N_d(\mu)x^\mu.
$$
We are now ready to state our main result which will be restated for the convenience of the reader in Propositions \ref{main1} and \ref{main2}:
\begin{theorem}\label{main}
It holds that
$
G_d(x) = G_1(x)^d(1-x)^{d-1} \text{ and }\mathcal G_d(x) = \mathcal G_1(x)^d,
$
where
$$
G_1(x) = \mathcal G_1(x) = \sum_{\ell=0}^\infty f(\ell+1)x^\ell.
$$
\end{theorem}

We will include explicit formulas for some specific growth functions $f$ related to usual sequences $\mathcal S$ which are used in practice. Such sequences are for example based on:
\begin{itemize}
\item Equidistant nodes without boundary:
$$
\mathring{\mathrm{E}}(n) = \left.\left\{\frac{2k}{n+1}-1\right| k\in\{1,\ldots,n\}\right\} , n\in \N_{\geqslant 1}
$$
For these points, the sequence given by $S_k = \mathring{\mathrm E}(f(k))$ is nested if $f(k) = m^k-1$ for $m\in\N_{\geqslant 2}$.

\item Equidistant nodes with boundary and Chebyshev nodes of the second kind:
$$
\mathrm{E}(n) = \left.\left\{\frac{2(k-1)}{n-1}-1\right| k\in\{1,\ldots,n\}\right\}, n\in\N_{\geqslant 2}
$$
$$
\mathrm{C}_2(n) = \left.\left\{\cos\left(\frac{k-1}{n-1}\pi\right)\right| k\in\{1,\ldots,n\}\right\}, n\in\N_{\geqslant 2}
$$
The sequences $S_k = \mathrm{E}(f(k))$ and $S_k = \mathrm{C}_2(f(k))$ respectively become nested if $f(k)=m^k+1$, $m\in\N_{\geqslant 2}$. Sometimes it is convenient to define $\mathrm E(1) = \mathrm C_{2}(1) = \{0\}$.

\item Chebyshev nodes of the first kind:
$$
\mathrm{C}_1(n) = \left.\left\{\cos\left(\frac{2k-1}{2n}\pi\right)\right| k\in\{1,\ldots,n\}\right\}, n\in\N_{\geqslant 1}
$$
Here, the sequence $S_k = \mathrm{C}_1(f(k))$ is nested if $f(k) = 3^k$ or $f(k)=3^{k-1}$.
\item Leja nodes (see \cite{Leja1957}): Any sequence $(x_n)_{n\in\mathbb N_{\geqslant 1}}$ such that $x_1\in [-1,1]$ and
$$
x_n \in\argmax\limits_{x\in [-1,1]} \prod_{j=1}^{n-1}|x-x_j|, n>1
$$
is called a Leja-sequence. For such a sequence, we let
$$
\operatorname{L}(n) = \left\{x_k|k\in\{1,\ldots,n\}\right\}.
$$
Note that $S_k = \mathrm{L}(f(k))$ is nested for every choice of a non-decreasing growth function $f$, in particular, for $f(k)=k$. There is a symmetrized version of Leja nodes available which become nested for the growth function $f(k) = 2k-1$ (see \cite{bos2015application}).

%reference 
%[1] Franciszek Leja. Sur certaines suites liées aux ensembles plans et leur application à la représentation conforme. In Annales Polonici Mathematici, volume 1, pages 8–13, 1957.

%Random Leja points
%Camille Pouchol

\end{itemize}

% You can generate these figure with the file \3_Mathematica\ZS_Grid.nb which will call our MATLAB module for creating the grid. Please just adjust the line : 
% SetDirectory["C:\\Users\\jocelyn.minini\\switchdrive\\MetaG\\6_Publications\\1_Counting\\Figures"] to save the figs in your own machine

\begin{example}
Consider a Smolyak grid based on the sequence $\mathcal S$, where $S_k = \mathrm C_1(3^k)$. As we show below, according to \eqref{chebyshev_new}, one obtains
$$\begin{aligned}
\mathrm N_2(\mu) & = 3^{\mu+1}(2\mu+3)\\
\mathrm N_3(\mu) & = 3^{\mu+1}(2\mu^2+10\mu+9)
\end{aligned}$$
so that the cardinalities in Figure \ref{fig:Grid} can be computed.

\begin{figure}[H]
\subfloat[\label{fig:Grid_2D_0}$\Gamma_2(0)$]{\includegraphics[width = 0.33\linewidth]{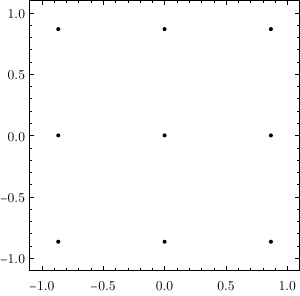}}\hfill
\subfloat[\label{fig:Grid_2D_1}$\Gamma_2(1)$]{\includegraphics[width = 0.33\linewidth]{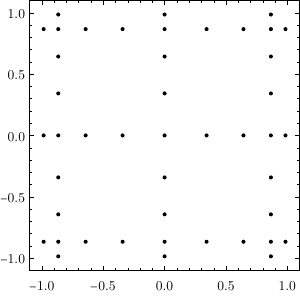}}\hfill
\subfloat[\label{fig:Grid_2D_2}$\Gamma_2(2)$]{\includegraphics[width = 0.33\linewidth]{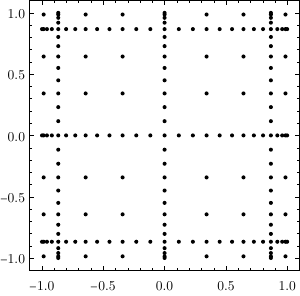}}\\
\subfloat[\label{fig:Grid_2D_0}$\Gamma_3(0)$]{\includegraphics[width = 0.33\linewidth]{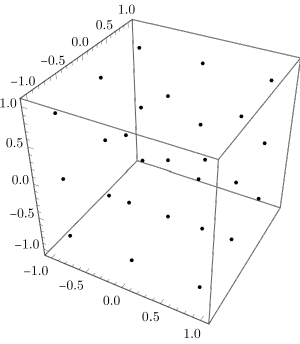}}\hfill
\subfloat[\label{fig:Grid_2D_1}$\Gamma_3(1)$]{\includegraphics[width = 0.33\linewidth]{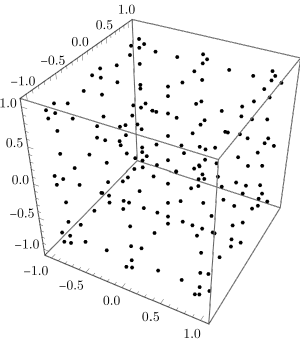}}\hfill
\subfloat[\label{fig:Grid_2D_2}$\Gamma_3(2)$]{\includegraphics[width = 0.33\linewidth]{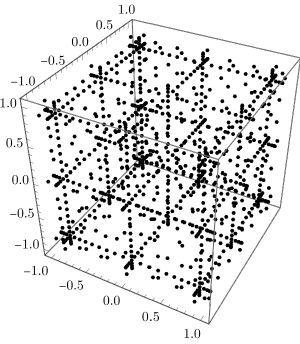}}
\caption{Smolyak grids with cardinality for $d=2$: (a) $9$, (b) $45$ (c) $189$ and in $d=3$: (d) $27$, (e) $189$, (f) $999$ respectively.}
\label{fig:Grid}
\end{figure}

%\begin{figure}[h]
%\begin{center}
%\includegraphics[scale=0.43]{smolyak_chebyshev_mu1}
%\includegraphics[scale=0.43]{smolyak_chebyshev_mu2}
%\includegraphics[scale=0.43]{smolyak_chebyshev_mu3}
%\caption{Smolyak grids $\Gamma(2,1)$, $\Gamma(2,2)$ and $\Gamma(2,3)$ with cardinality $45,189$ and $729$ respectively.}\label{fig_smolyak_2d}
%\end{center}

%\end{figure}

%\begin{figure}[h]
%\begin{center}
%\includegraphics[scale=0.34]{smolyak_chebyshev_mu0-3d}
%\includegraphics[scale=0.34]{smolyak_chebyshev_mu1-3d}
%\includegraphics[scale=0.34]{smolyak_chebyshev_mu2-3d}
%\caption{Smolyak grids  $\Gamma(3,0)$, $\Gamma(3,1)$ and $\Gamma(3,2)$ with cardinality $27,189$ and $999$ respectively.}\label{fig_smolyak_3d}
%\end{center}

%\end{figure}
\end{example}

%In this paper we will provide explicit formulas for the numbers
%$\mathrm N(d,\mu)$, $\mathcal N(d,\mu)$ and $\mathcal N_\Sigma(d,\mu)$, where we focus on the growth functions
%$$
%\begin{aligned}
%k&\mapsto n^k-1\\
%k& \mapsto n^k\\
%k&\mapsto n^k+1\\
%k & \mapsto k
%\end{aligned}
%$$
%Was sind gute growth functions, resp.\ solche die auftreten in der Praxis?
\section{Counting Nodes}
We first consider a nested sequence $\mathcal S$ with growth function $f(k)=|S_k|$. Then we have the following proposition:
\begin{prop}\label{main1}
The generating function
$$
G_d(x) = \sum_{\mu=0}^\infty \mathrm N_d(\mu)x^\mu
$$
satisfies
$G_d(x) = G_1(x)^d(1-x)^{d-1},$ where $\displaystyle G_1(x) = \sum_{\ell=0}^\infty f(\ell+1)x^\ell.$
\end{prop}

\begin{proof}
We have
$$\begin{aligned}
\Gamma_{d+1}(\mu) & = \bigcup_{|\mathbf i| = d+1+\mu}\bigtimes_{k=1}^{d+1} S_{i_k}
 = \bigcup_{\ell = 0}^{\mu} \Gamma_d(\mu-\ell) \times S_{\ell+1}\\
& = \Gamma_d(\mu)\times S_{1}\cup \bigcup_{\ell = 1}^{\mu}  \Gamma_d(\mu-\ell) \times S_{\ell+1}\\
& =\Gamma_d(\mu)\times S_{1}\cup \bigcup_{\ell = 1}^{\mu}  \Gamma_d(\mu-\ell) \times \left(S_{\ell+1}\setminus S_{\ell}\right),
\end{aligned}
$$
where the key point is that the last line writes $\Gamma_{d+1}(\mu)$ as a \emph{disjoint} union because of the nestedness of $\mathcal S$. This idea already appears in \cite{muller-gronbach_hyperbolic_1998,novak_simple_1999,ullrich}. It follows that
\begin{equation}\label{recursion}
\mathrm N_{d+1}(\mu) = f(1) \cdot \mathrm N_d(\mu) + \sum_{\ell=1}^\mu (f(\ell+1)-f(\ell))\mathrm N_d(\mu-\ell),
\end{equation}
from which we deduce using the Cauchy product, since $\mathrm N_d(k)=0$ provided $k<0$,
$$
\begin{aligned}
G_{d+1}(x) & = f(1) \cdot G_d(x) +\sum_{\mu=1}^\infty \left(\sum_{\ell=0}^{\mu-1}(f(\ell+2)-f(\ell+1)) \mathrm N_d(\mu-1-\ell)\right)x^\mu\\
& = f(1)\cdot G_d(x) + x \sum_{\mu=0}^\infty\left(\sum_{\ell=0}^{\mu}(f(\ell+2)-f(\ell+1)) \mathrm N_d(\mu-\ell)\right)x^{\mu}\\
& = f(1)\cdot G_d(x) + x \left(\sum_{\ell=0}^\infty(f(\ell+2)-f(\ell+1))x^\ell\right)\left(\sum_{\mu=0}^\infty \mathrm N_d(\mu)x^\mu\right).
\end{aligned}
$$
Since 
$$
\sum_{\ell=0}^\infty(f(\ell+2)-f(\ell+1))x^\ell = G_1(x)-f(1)-x\, G_1(x),
$$
we conclude that
$$
\begin{aligned}
G_{d+1}(x) & = f(1)\cdot G_d(x) + \left(G_1(x)-f(1)-x \,G_1(x)\right)G_d(x)
\\ & = G_d(x)G_1(x)(1-x),
\end{aligned}
$$
from which the claim follows by induction on $d$.
\end{proof}

Let now $\mathcal S$ be a (not necessarily nested) sequence with growth function $f$ and consider the generating function
$$
\mathcal G_d(x) = \sum_{\mu=0}^\infty \mathcal N_d(\mu)x^\mu.
$$
\begin{prop}\label{main2}
It holds that $
\mathcal G_d(x)= \mathcal G_1(x)^d$, where
$\displaystyle \mathcal G_1(x) = \sum_{\mu = 0}^\infty f(\mu+1)x^\mu$.
\end{prop}

\begin{proof}
We start by observing that
$$\begin{aligned}
\mathcal N_{d+1}(\mu)& = \sum_{|\mathbf i|=d+1+\mu}\prod_{k=1}^{d+1}f(i_k) = \sum_{|\mathbf i|=d+1+\mu}\left(\prod_{k=1}^{d}f(i_k)\right) f(i_{d+1})\\
& = \sum_{\ell=1}^{\mu+1}f(\ell)\mathcal N_d(\mu+1-\ell)
\end{aligned}$$
so that using the Cauchy product
$$\begin{aligned}
\mathcal G_{d+1}(x) & = \sum_{\mu=0}^\infty \sum_{\ell=0}^{\mu}f(\ell+1)\mathcal N_d(\mu-\ell)x^\mu
 =\left(\sum_{\ell=0}^{\infty}f(\ell+1)x^{\ell}\right)\left(\sum_{\mu=0}^\infty\mathcal N_d(\mu)x^{\mu}\right)\\
& = \mathcal G_1(x)\mathcal G_d(x)\end{aligned}$$
and hence the claim follows by induction on $d$.\end{proof}

%Note that the formulas for $\mathcal N(d,\mu,f)$ that are obtained using $\mathcal G_d$ count the number of nodes in the corresponding grid in the case where the sequence $\mathcal S$ is \emph{strictly non-nested}, i.e.\ if $S_k\cap S_j=\emptyset$ for all $k,j\in \mathbb N_{\geqslant 1}$ such that $k\ne j$.

When discussing the applications of our main result, we will focus on nested sequences $\mathcal S$ but note that the formulas $\mathcal N_d(\mu)$ and $\mathcal N^\Sigma_d(\mu)$ do not depend on the nestedness or non-nestedness of $\mathcal S$. The general recipe for obtaining an expression for $\mathrm N_d(\mu)$ or $\mathcal N_d(\mu)$ consists in evaluating $G_d$ and $\mathcal G_d$ respectively and then -- for the growth functions under consideration --  one can use Newton's binomial theorem and the Cauchy product in order read off $\mathrm N_d(\mu)$ or $\mathcal N_d(\mu)$. We will provide the details for the first three growth functions under consideration -- the others being obtained similarly:

\subsection{Growth Function $f(k) = n^{k}-1$}

In this case, $\displaystyle
G_1(x) = \frac{n-1}{(1-x)(1-nx)},
$
so that using Newton's binomial theorem
$$\begin{aligned}
G_d(x) & = (n-1)^d\left(\frac{1}{1-nx}\right)^d\frac{1}{1-x}\\
& = (n-1)^d\left(\sum_{k=0}^\infty\binom{d+k-1}{k}(n x)^k\right)\left(\sum_{\ell=0}^\infty x^\ell\right).
\end{aligned}$$
Reading off the coefficient of $x^\mu$ using the Cauchy product yields
$$
\mathrm N_d(\mu) =(n-1)^d \sum_{k=0}^\mu\binom{d+k-1}{k}n^k.
$$
\begin{remark}
If $n=2$, one can recover the formula given by \citet[Lemma 3.6]{Bungartz_Griebel_2004} by exploting the fact that
$$
\binom{d-1+k}{k}=\binom{d-1+k}{d-1}.
$$
and observing that there, the formula is stated for $\mu-1$.
\end{remark}
Smiliarly we obtain
$$\begin{aligned}
\mathcal G_d(x) & = (n-1)^d \left(\frac{1}{1-x}\right)^d\left(\frac{1}{1-nx}\right)^d\\
& = (n-1)^d \left(\sum_{k=0}^\infty \binom{d+k-1}{k}(nx)^k\right)\left(\sum_{\ell=0}^\infty  \binom{d+\ell-1}{\ell}x^\ell\right).
\end{aligned}$$
The coefficient of the term $x^\mu$ is therefore $$
\mathcal N_d(\mu) = (n-1)^d\sum_{k=0}^\mu \binom{d+k-1}{k}\binom{d+\mu-k-1}{\mu-k}n^k.
$$
\subsection{Growth Function $f(k) = n^k$}
Here,
$\displaystyle
G_1(x) = \frac{n}{1-nx},
$
so that
$$\begin{aligned}
G_d(x) & = n^d \left(\frac{1}{1-nx}\right)^d(1-x)^{d-1}\\
& = n^d \left(\sum_{k=0}^{d-1} \binom{d-1}{k}(-1)^kx^k\right)\left(\sum_{\ell=0}^\infty\binom{d+\ell-1}{\ell}(n x)^\ell\right).
\end{aligned}$$
Reading off the coefficient of $x^\mu$ leads to
\begin{equation}\label{chebyshev_new}
\mathrm N_d(\mu) =n^d\sum_{k=0}^{\min\{d-1,\mu\}}\binom{d-1}{k}(-1)^k\binom{d+\mu-k-1}{\mu-k}n^{\mu-k},
\end{equation}
which can be shown to be equivalent to
\begin{equation}\label{chebyshev_new_2}
\mathrm N_d(\mu) = \sum_{k=0}^{\min\{d-1,\mu\}}\binom{d-1}{k}\binom{\mu}{k}n^{\mu+d-k}(n-1)^{k}.
\end{equation}

\begin{remark}
\begin{enumerate}

\item If one takes $f(k)=n^{k-1}$ in \eqref{chebyshev_new_2}, then one obtains the same result up to a factor $n^d$ so that in this case, if $n=2$, one obtains

$$
\mathrm N_d(\mu) = \sum_{k=0}^{\min\{d-1,\mu\}}\binom{d-1}{k}\binom{\mu}{k}2^{\mu-k},
$$

which equals exactly the formula found by \citet[Lemma 2]{ullrich}.
\item Formula \eqref{chebyshev_new_2} with $n=3$ gives an explicit expression for $\mathrm N_d(\mu)$ for the case of a Smolyak-grid constructed with Chebyshev points of the first kind. If $f(k)=3^{k-1}$ is used instead, one obtains the same formula up to a factor $3^d$.
\end{enumerate}
\end{remark}
Similarly,
$$
\mathcal G_d(x) = n^d\left(\frac{1}{1-n x}\right)^d = n^d\left(\sum_{k=0}^\infty \binom{d+k-1}{k}(nx)^k\right)
$$
so that the coefficient of $x^\mu$ can be read off and one obtains
$$
\mathcal N_d(\mu) = n^{d+\mu}\binom{d+\mu-1}{\mu}.
$$
\begin{remark}
This formula can be obtained in a much simpler way: If $f(k) = n^k$ one has
$$
\mathcal N_d(\mu) = \sum_{|\mathbf i|=d+\mu}n^{|\mathbf i|} = n^{d+\mu}\left.\left|\left\{\mathbf i\in \mathbb N_{\geqslant 1}^d\right| |\mathbf i|=d+\mu\right\}\right|
$$
and the equality$$
\left.\left|\left\{\mathbf i\in \mathbb N_{\geqslant 1}^d\right| |\mathbf i|=d+\mu\right\}\right|=\binom{d+\mu-1}{\mu}$$
can be obtained using a stars and bars argument.
\end{remark}
\subsection{Growth Function $f(k) = n^{k}+1$}
Here,
$\displaystyle
G_1(x) = \frac{n+1-2nx}{(1-x)(1-nx)},
$
so that
%$$\begin{aligned}
%G_d(x) & = (n+1-2nx)^d\left(\frac{1}{1-nx}\right)^d\frac{1}{1-x}\\
%& = \left(\sum_{k=0}^d\binom{d}{k}(-2n)^k(n+1)^{d-k}x^k\right)\left(\sum_{i=0}^\infty\binom{d-1+i}{i}(n x)^i\right)\left(\sum_{j=0}^\infty x^j\right)\\
%& =  \left(\sum_{k=0}^d\binom{d}{k}(-2n)^k(n+1)^{d-k}x^k\right)\left(\sum_{\ell=0}^\infty\left( \sum_{j=0}^\ell \binom{d-1+j}{j}n^j\right)x^\ell\right).
%\end{aligned}$$
%Reading off the coefficient of $x^\mu$ yields
$$
\mathrm N_d(\mu)=\sum_{k=0}^{\min\{d,\mu\}} \binom{d}{k}(-2n)^k(n+1)^{d-k} \sum_{\ell=0}^{\mu-k}\binom{d+\ell-1}{\ell}n^\ell.
$$
\begin{remark}
Note that for the case $n=2$ and $f(k)=2^k+1$, a formula is already available in \cite{Bungartz_Griebel_2004} which thinks of this case as upgrading the case of the growth function $g(k)=2^k-1$ by two boundary points. Counting the $j$-skeleta of the $d$-dimensional hypercube yields $\binom{d}{j}2^{d-j}$ for fixed $j$. Then the formula for $f$ is obtained by building the sum of the formula for $g$ over all $j$-skeleta as $j=0,\ldots, d$. This approach -- for general $n$ -- yields the equivalent formula
$$
\mathrm N_d(\mu)=\sum_{k=0}^{d} \binom{d}{k}2^{d-k}(n-1)^k\sum_{\ell=0}^\mu \binom{k+\ell-1}{k-1}n^\ell.
$$

\end{remark}
Furthermore, in this case,
$\displaystyle
\mathcal G_d(x) = (n+1-2nx)^d\left(\frac{1}{1-nx}\right)^d\left(\frac{1}{1-x}\right)^d,
$
so that
%$$\begin{aligned}
%\mathcal G_d(x) & = \sum_{k=0}^d\binom{d}{k}(-2nx)^k(n+1)^{d-k}\left(\sum_{i=0}^\infty \binom{d+i-1}{i}(nx)^i\right)\left(\sum_{j=0}^\infty  \binom{d+j-1}{j}x^j\right)\\
%& = \sum_{k=0}^d\binom{d}{k}(-2nx)^k(n+1)^{d-k}\left(\sum_{\nu=0}^\infty \sum_{\ell=0}^\nu \binom{d+\ell-1}{\ell}\binom{d+\nu-\ell-1}{\nu-\ell}n^\ell x^\nu\right)\end{aligned}
%$$
%Reading off the coefficient of $x^\mu$ leads to
$$
\mathcal N_d(\mu) =\!\!\!
\sum_{\ell = 0}^{\min\{d,\mu\}}\sum_{k=0}^{\mu-\ell} \binom{d}{\ell}(-1)^\ell (2n)^\ell(n+1)^{d-\ell}\binom{d+k-1}{k}\binom{d+\mu-\ell-k-1}{\mu-\ell-k}n^k.
$$

\subsection{Growth Function $f(k) = k$}
Here, $G_1(x) = \frac{1}{(1-x)^2}$ so that
$\displaystyle
G_d(x) = \frac{1}{(1-x)^{d+1}}= \sum_{k=0}^\infty \binom{d+k}{k}x^k
$
and hence
$$
\mathrm{N}_d(\mu) = \binom{d+\mu}{\mu}.
$$
Similarly, we obtain
$$
\mathcal G_d(x) = \frac{1}{(1-x)^{2d}}= \sum_{k=0}^\infty \binom{2d+k-1}{k}x^k
$$
and hence
$$
\mathcal{N}_d(\mu) = \binom{2d+\mu-1}{\mu}.
$$
In this case, we obtain setting $m=\max\{0,\mu+1-d\}$:
$$
\mathcal{N}^{\Sigma}_d(\mu) =\frac{1}{2d}\left((1+\mu)\binom{2d+\mu}{\mu+1}-m\binom{2d+m-1}{m}\right)
$$

\subsection{Further Remarks}

If $f(k) = 2^{k-1}+1$ for $k>1$ and $f(1)=1$, one can show by induction on $d$ using~\eqref{recursion} that
$\mathrm N_d(\mu)-1 = 2^{\mu} P_{d-1}(\mu),
$
where $P_{d-1}$ is a polynomial of degree $d-1$ with leading coefficient $1/({2^{d-1}(d-1)!})$ so that one obtains the asymptotics
$$
\mathrm N_d(\mu) \sim\color{black} \frac{2^{\mu}}{2^{d-1}(d-1)!}\mu^{d-1}
$$
for $d$ fixed and $\mu\to\infty$, compare \citet[Lemma 1]{muller-gronbach_hyperbolic_1998}.

Using $f(k) = 2k-1$ yields
$\displaystyle
G_1(x) = \frac{1+x}{(1-x)^2}
$
so that
$\displaystyle
G_d(x) = \frac{(1+x)^d}{(1-x)^{d+1}}
$
and hence
$$
\mathrm N_d(\mu) = \sum_{k=0}^{\min\{d,\mu\}}\binom{\mu}{k}\binom{\mu+d-k}{\mu}
$$
which is precisely Theorem 2 of \cite{novak_simple_1999}. We note that this formula is also obtained using generating functions.

\section*{Acknowledgments}

This study has been financed by the ``Ingénierie et Architecture'' domain
of HES-SO, University of Applied Sciences Western Switzerland,
which is acknowledged. We further thank the referees for their careful reading and their valuable comments and suggestions.


\begin{thebibliography}{12}
\expandafter\ifx\csname natexlab\endcsname\relax\def\natexlab#1{#1}\fi
\providecommand{\url}[1]{\texttt{#1}}
\providecommand{\href}[2]{#2}
\providecommand{\path}[1]{#1}
\providecommand{\DOIprefix}{doi:}
\providecommand{\ArXivprefix}{arXiv:}
\providecommand{\URLprefix}{URL: }
\providecommand{\Pubmedprefix}{pmid:}
\providecommand{\doi}[1]{\href{http://dx.doi.org/#1}{\path{#1}}}
\providecommand{\Pubmed}[1]{\href{pmid:#1}{\path{#1}}}
\providecommand{\bibinfo}[2]{#2}
\ifx\xfnm\relax \def\xfnm[#1]{\unskip,\space#1}\fi
%Type = Article
\bibitem[{Bos and Caliari(2015)}]{bos2015application}
\bibinfo{author}{Bos, L.}, \bibinfo{author}{Caliari, M.}, \bibinfo{year}{2015}.
\newblock \bibinfo{title}{Application of modified leja sequences to polynomial
  interpolation}.
\newblock \bibinfo{journal}{Dolomites Research Notes on Approximation}
  \bibinfo{volume}{8}.
%Type = Article
\bibitem[{Bungartz and Griebel(2004)}]{Bungartz_Griebel_2004}
\bibinfo{author}{Bungartz, H.J.}, \bibinfo{author}{Griebel, M.},
  \bibinfo{year}{2004}.
\newblock \bibinfo{title}{Sparse grids}.
\newblock \bibinfo{journal}{Acta Numerica} \bibinfo{volume}{13},
  \bibinfo{pages}{147–269}.
\newblock \DOIprefix\doi{10.1017/S0962492904000182}.
%Type = Article
\bibitem[{Burkardt(2014)}]{burkardt}
\bibinfo{author}{Burkardt, J.}, \bibinfo{year}{2014}.
\newblock \bibinfo{title}{Counting abscissas in sparse grids} \URLprefix
  \url{https://people.math.sc.edu/Burkardt/presentations/sgmga_counting.pdf}.
%Type = Article
\bibitem[{Judd et~al.(2014)Judd, Maliar, Maliar and Valero}]{judd}
\bibinfo{author}{Judd, K.L.}, \bibinfo{author}{Maliar, L.},
  \bibinfo{author}{Maliar, S.}, \bibinfo{author}{Valero, R.},
  \bibinfo{year}{2014}.
\newblock \bibinfo{title}{Smolyak method for solving dynamic economic models:
  Lagrange interpolation, anisotropic grid and adaptive domain}.
\newblock \bibinfo{journal}{Journal of Economic Dynamics and Control}
  \bibinfo{volume}{44}, \bibinfo{pages}{92--123}.
\newblock \URLprefix
  \url{https://www.sciencedirect.com/science/article/pii/S0165188914000621},
  \DOIprefix\doi{https://doi.org/10.1016/j.jedc.2014.03.003}.
%Type = Article
\bibitem[{Kaarnioja(2013)}]{kaarnioja_smolyak_2013}
\bibinfo{author}{Kaarnioja, V.}, \bibinfo{year}{2013}.
\newblock \bibinfo{title}{Smolyak quadrature}.
\newblock \bibinfo{journal}{University of Helsinki, Department of Mathematics
  and Statistics} \URLprefix
  \url{https://helda.helsinki.fi/bitstreams/a2d5bfe3-d2da-4d9e-b536-6383d99b5486/download}.
%Type = Article
\bibitem[{Leja(1957)}]{Leja1957}
\bibinfo{author}{Leja, F.}, \bibinfo{year}{1957}.
\newblock \bibinfo{title}{Sur certaines suites liées aux ensembles plans et
  leur application à la représentation conforme}.
\newblock \bibinfo{journal}{Annales Polonici Mathematici} \bibinfo{volume}{4},
  \bibinfo{pages}{8--13}.
\newblock \URLprefix \url{http://eudml.org/doc/208291}.
%Type = Article
\bibitem[{Müller-Gronbach(1998)}]{muller-gronbach_hyperbolic_1998}
\bibinfo{author}{Müller-Gronbach, T.}, \bibinfo{year}{1998}.
\newblock \bibinfo{title}{Hyperbolic cross designs for approximation of random
  fields}.
\newblock \bibinfo{journal}{Journal of Statistical Planning and Inference}
  \bibinfo{volume}{66}, \bibinfo{pages}{321--344}.
\newblock \URLprefix
  \url{https://www.sciencedirect.com/science/article/pii/S0378375897000888},
  \DOIprefix\doi{10.1016/S0378-3758(97)00088-8}.
%Type = Article
\bibitem[{Novak and Ritter(1999)}]{novak_simple_1999}
\bibinfo{author}{Novak, E.}, \bibinfo{author}{Ritter, K.},
  \bibinfo{year}{1999}.
\newblock \bibinfo{title}{Simple {Cubature} {Formulas} with {High} {Polynomial}
  {Exactness}}.
\newblock \bibinfo{journal}{Constructive Approximation} \bibinfo{volume}{15},
  \bibinfo{pages}{499--522}.
\newblock \URLprefix \url{https://doi.org/10.1007/s003659900119},
  \DOIprefix\doi{10.1007/s003659900119}.
%Type = Article
\bibitem[{Sickel and Ullrich(2007)}]{ullrich}
\bibinfo{author}{Sickel, W.}, \bibinfo{author}{Ullrich, T.},
  \bibinfo{year}{2007}.
\newblock \bibinfo{title}{The smolyak algorithm, sampling on sparse grids and
  function spaces of dominating mixed smoothness} .
%Type = Article
\bibitem[{Smolyak(1963)}]{smolyak}
\bibinfo{author}{Smolyak, S.A.}, \bibinfo{year}{1963}.
\newblock \bibinfo{title}{Quadrature and interpolation formulas for tensor
  products of certain classes of functions}.
\newblock \bibinfo{journal}{Dokl. Akad. Nauk SSSR} \bibinfo{volume}{148},
  \bibinfo{pages}{1042--1045}.
%Type = Article
\bibitem[{Wasilkowski and Wozniakowski(1995)}]{wasilkowski_explicit_1995}
\bibinfo{author}{Wasilkowski, G.W.}, \bibinfo{author}{Wozniakowski, H.},
  \bibinfo{year}{1995}.
\newblock \bibinfo{title}{Explicit {Cost} {Bounds} of {Algorithms} for
  {Multivariate} {Tensor} {Product} {Problems}}.
\newblock \bibinfo{journal}{Journal of Complexity} \bibinfo{volume}{11},
  \bibinfo{pages}{1--56}.
\newblock \URLprefix
  \url{https://www.sciencedirect.com/science/article/pii/S0885064X85710011},
  \DOIprefix\doi{10.1006/jcom.1995.1001}.
%Type = Article
\bibitem[{Xu(2015)}]{xu_weak_2015}
\bibinfo{author}{Xu, G.}, \bibinfo{year}{2015}.
\newblock \bibinfo{title}{On weak tractability of the {Smolyak} algorithm for
  approximation problems}.
\newblock \bibinfo{journal}{Journal of Approximation Theory}
  \bibinfo{volume}{192}, \bibinfo{pages}{347--361}.
\newblock \URLprefix
  \url{https://www.sciencedirect.com/science/article/pii/S0021904514002056},
  \DOIprefix\doi{10.1016/j.jat.2014.10.016}.

\end{thebibliography}
\end{document}